\documentclass[11pt]{amsart}
\usepackage{graphicx}
\usepackage{amssymb, amsmath}
\newtheorem{theorem}{Theorem}[section]
\newtheorem{corollary}[theorem]{Corollary}
\newtheorem{lemma}[theorem]{Lemma}
\newtheorem{proposition}[theorem]{Proposition}
\theoremstyle{definition}
\newtheorem{definition}[theorem]{Definition}

\theoremstyle{remark}

\numberwithin{equation}{section}

\newcommand{\Gf}{der_{\theta, b}(G, f)}

\begin{document}
\title{Simple polyadic groups}
\author{H. Khodabandeh}
\address{H. Khodabandeh\\
Department of Pure Mathematics,  Faculty of Mathematical
Sciences, University of Tabriz, Tabriz, Iran }
\author{\sc M. Shahryari}
\thanks{{\scriptsize
\hskip -0.4 true cm MSC(2010): 20N15
}}

\address{M. Shahryari\\
 Department of Pure Mathematics,  Faculty of Mathematical
Sciences, University of Tabriz, Tabriz, Iran }

\email{mshahryari@tabrizu.ac.ir}
\date{\today}

%%% ----------------------------------------------------------------------
\begin{abstract}
The main aim of this article is to establish a classification of
simple polyadic groups in terms of ordinary groups and their
automorphisms. We give two different definitions of simpleness for
polyadic groups, from the point of views of universal algebra, UAS
(universal algebraically simpleness), and group theory, GTS (group
theoretically simpleness). We obtain the necessary and sufficient
conditions for a polyadic group to be UAS or GTS.

\end{abstract}

\maketitle

%%%%%%%%%%%%%%%%%%%%%%%%%%%%%%%%%%%%%%%%%%%%%%%%%%%%%%%%%%%%%%%%%%%%%%

\section{Introduction}

Let $G$ be a non-empty set and $n$ be a positive integer. If
$f:G^n\to G$ is an $n$-ary operation, then we use the compact
notation $f(x_1^n)$ for the elements $f(x_1, \ldots, x_n)$. In
general, if $x_i, x_{i+1}, \ldots, x_j$ is an arbitrary sequence of
elements in $G$, then we denote it as $x_i^j$. In the special case,
when all terms of this sequence are equal to a constant $x$, we show
it by $\stackrel{(t)}{x}$, where $t$ is the number of terms. During
this article, we assume that $n\geq 2$. We say that an $n$-ary
operation is {\em associative}, if for any $1\leq i<j\leq n$, the
equality
$$
f(x_1^{i-1},f(x_i^{n+i-1}),x_{n+i}^{2n-1})=
f(x_1^{j-1},f(x_j^{n+j-1}),x_{n+j}^{2n-1})
$$
holds for all $x_1,\ldots,x_{2n-1}\in G$. An $n$-ary system $(G,f)$
is called an {\em $n$-ary group} or a {\em polyadic group}, if $f$
is associative and for all $a_1,\ldots,a_n, b\in G$ and $1\leq i\leq
n$, there exists a unique element $x\in G$ such that
$$
f(a_1^{i-1},x,a_{i+1}^n)=b.
$$
It is proved that the uniqueness assumption on the solution $x$ can
be dropped, see \cite{Dud2} for details. Clearly, the case $n=2$ is
just the definition of ordinary groups.

For a review of basic notions, we introduce some materials. First of
all, there is the classical paper of E. Post \cite{Post}, which is
one of the first articles published on the subject. In this paper,
Post proves his well-known {\em Coset theorem}. Many basic
properties of polyadic groups are studied in this paper. The
articles \cite{Dor} and \cite{Kas} are among the first materials
written on the polyadic groups. Reader who knows Russian, can use
the book of Galmak \cite{Gal}, for an almost complete description of
polyadic groups and as we understand from its English abstract, some
of our results are also proved in that book by a different method.
The articles \cite{Art}, \cite{Dud1}, \cite{Gal2}, and \cite{Gle}
can be used for study of axioms of polyadic groups as well as their
varieties.

Note that an $n$-ary system $(G,f)$ of the form $\,f(x_1^n)=x_1
x_2\ldots x_nb$, where $(G,\cdot)$ is a group and $b$  a fixed
element belonging to the center of $(G,\cdot)$, is an $n$-ary group.
Such an $n$-ary group is called {\em $b$-derived} from the group
$(G,\cdot)$ and it is denoted by $der_b(G, \cdot)$. In the case when $b$ is the identity of $(G,\cdot)$ we
say that such a polyadic group is {\em reduced} to the group $(G, \cdot
)$ or {\em derived} from $(G, \cdot)$ and we use the notation $der(G, \cdot)$ for it. But for every $n>2$, there
are $n$-ary groups which are not derived from any group. An $n$-ary
group $(G,f)$ is derived from some group if and only if it contains
an element $a$ (called an {\em $n$-ary identity}) such that
$$
 f(\stackrel{(i-1)}{a},x,\stackrel{(n-i)}{a})=x
$$
holds for all $x\in G$ and for all $i=1,\ldots,n$.

From the definition of an $n$-ary group $(G,f)$ we can directly see
that for every $x\in G$ there exists only one $y\in G$ satisfying
the equation
$$
f(\stackrel{(n-1)}{x},y)=x.
$$
This element is called {\em skew} to $x$ and it is denoted by
$\overline{x}$.    As D\"ornte proved (see a \cite{Dor}), the
following identities hold for all $\,x,y\in G$, $2\leq i,j\leq n$
and $1\leq k\leq n$
$$
f(\stackrel{(i-2)}{x},\overline{x},\stackrel{(n-i)}{x},y)=
f(y,\stackrel{(n-j)}{x},\overline{x},\stackrel{(j-2)}{x})=y,
$$
$$
f(\stackrel{(k-1)}{x},\overline{x},\stackrel{(n-k)}{x})=x.
$$

Suppose $(G, f)$ is a polyadic group and $a\in G$ is a fixed
element. Define a binary operation
$$
x\ast y=f(x,\stackrel{(n-2)}{a},y).
$$
Then $(G, \ast)$ is an ordinary group, called the {\em retract} of
$(G, f)$ over $a$. Such a retract will be denoted by $ret_a(G,f)$. All
retracts of a polyadic group are isomorphic, see \cite{DM}. The
identity of the group $(G,\ast)$ is $\overline{a}$. One can verify
that the inverse element to $x$ has the form
$$
y=f(\overline{a},\stackrel{(n-3)}{x},\overline{x},\overline{a}).
$$

One of the most fundamental theorems of polyadic group is the
following, now known as {\em Hossz\'{u} -Gloskin's theorem}. We will use
it frequently in this article and the reader can use \cite{DG},
\cite{DM1}, \cite{Hos} and \cite{Sok} for detailed discussions.

\begin{theorem}
Let $(G,f)$ be an $n$-ary group. Then

1. on $G$ one can define an operation $\cdot$ such
that $(G,\cdot)$ is a group,

2.  there exist an automorphism $\theta$ of $(G,\cdot)$ and $b\in G$, such
that $\theta(b)=b$,

3.  $\theta^{n-1}(x)=b x b^{-1}$, for every $x\in
G$,

4.  $f(x_1^n)=x_1\theta(x_2)\theta^2(x_3)\cdots\theta^{n-1}(x_n)b$, for
all $x_1,\ldots,x_n\in G$.
\end{theorem}

According to this theorem, we  use the notation $\Gf$ for $(G,f)$
and we say that $(G,f)$ is $(\theta, b)$-derived from the group $(G,
\cdot)$. During this paper, we will assume that $(G, f)=\Gf$.

Varieties of polyadic groups and the structure of congruences on
polyadic groups are studied in \cite{Art}, \cite{Dud1} and
\cite{Usan}. It is proved that all congruences on polyadic groups
are commute and so the lattice of congruences is modular. In this
article we will give a very simple proof for this fact. Among other
issues, investigated on polyadic groups, is the representation
theory. In \cite{Dud-Shah} the representation theory of polyadic
groups is studied and \cite{Shah} contains generalizations of some
important theorems of character theory of finite groups to the case
of polyadic groups.

We established some  fundamental results on the structure of
homomorphisms of polyadic groups in \cite{Khod-Shah}. We will use
the main result of that article here as a very strong tool. Suppose
$(H, \ast)$ is an ordinary group and $a\in H$. In the following
theorem, we denote the inner automorphism $x\mapsto a\ast x\ast
a^{-1}$ by $I_a$. Also $R_a$ denotes the map $x\mapsto x\ast a$.

\begin{theorem}
Suppose $(G,f)=der_{\theta, b}(G, \cdot)$ and $(H, h)=der_{\eta,
c}(H, \ast)$ are two polyadic groups. Let $\psi: (G,f)\to (H, h)$ be
a homomorphism. Then there exists $a\in H$ and an ordinary
homomorphism $\phi:
(G, \cdot)\to (H, \ast)$, such that $\psi=R_a\phi$. Further $a$ and $\phi$ satisfy the following conditions;\\
$$
h(\stackrel{(n)}{a})=\phi(b)\ast a\ \ \ and \ \ \
\phi\theta=I_{a}\eta\phi.
$$
Conversely, if $a$ and $\phi$ satisfy the above two conditions, then
$\psi=R_a\phi$ is a homomorphism $(G,f)\to  (H,h)$.
\end{theorem}

Clearly, using the above theorem, we can determine when two polyadic
groups $(G,f)=der_{\theta, b}(G, \cdot)$ and $(H, h)=der_{\eta,
c}(H, \ast)$ are isomorphic,  and also, it can be applied for a
complete description of polyadic representations, see
\cite{Khod-Shah} for details. Using this theorem, we will determine
the structure of polyadic subgroups in the section 4.

Before going to explanation of the motivations for the recent work,
we recall the definition of normal polyadic subgroups from
\cite{Dud-Shah}. An $n$-ary subgroup $H$ of a polyadic group $(G,f)$
is called {\em normal} if
$$
f(\overline{x},\stackrel{(n-3)}{x},h,x)\in H
$$
for all $h\in H$ and $x\in G$. If every normal subgroup of $(G, f)$
is singleton or equal to $G$, then we say that $(G, f)$ is {\em
group theoretically simple} or it is $GTS$ for short. If $H=G$ is
the only normal subgroup of $(G, f)$, then we say it is {\em
strongly simple in the group theoretic sense} or $GTS^{\ast}$ for
short. For any normal subgroup $H$ of an $n$-ary group $(G,f)$, we
define the relation $\sim_H$ on $G$, by

$$
 x\sim_H y\ \Longleftrightarrow\ \exists h_1, \ldots, h_{n-1}\in H:\
y=f(x,h_1^{n-1}).
$$
Now, it is easy to see that such
defined relation is an equivalence on $G$. The equivalence class of
$G$, containing $x$ is denoted by $xH$ and is called the {\it left
coset} of $H$ with the representative $x$. On the set $G/H=\{xH: x\in
G\}$, we introduce the operation
$$
f_H(x_1H,x_2H,\ldots,x_nH)=f(x_1^n)H.
$$
Then $(G/H,f_H)$ is an $n$-ary group derived from the group
$ret_H(G/H,f_H)$, see \cite{Dud-Shah}. One of the main aims of this
article is to classify all $GTS$ polyadic groups. We will give a
necessary and sufficient condition for a polyadic group $(G, f)$ to
be $GTS$ in terms of the ordinary group $(G, \cdot)$ and the
automorphism $\theta$. It is possible to define another kind of
simpleness for polyadic groups,  universal algebraically simpleness.
Note that an equivalence relation $R$ over $G$ is said to be a {\em
congruence}, if

1. $\forall i:\ x_iRy_i\Rightarrow f(x_1^n)Rf(y_1^n)$,

2. $xRy\Rightarrow \overline{x}R\overline{y}$.

For example, if $H$ is a normal polyadic subgroup of $(G, f)$, then
$R=\sim_H$ is a congruence, see \cite{Dud-Shah}. We say that $(G, f)$ is {\em universal
algebraically simple} or $UAS$ for short, if the only congruence is
the {\em equality} and $G\times G$. We also, will give a
classification of $UAS$ polyadic groups and we will prove that
$UAS\subseteq GTS$.

Our motivation to  study of simple polyadic groups came from $n$-Lie
algebras (or Filippov algebras). A vector space $L$ over a field
$\mathbb{F}$ is an $n$-ary Lie algebra, if it is equipped with an
alternative $n$-linear map $[-, \cdots,-]:L^n\to L$ such that the
following {\em Jacobi identity} holds
$$
\sum_{i=1}^n[x_1, \ldots, x_{i-1}, [x_i, y_1, \ldots, y_{n-1}],
x_{i+1}, \ldots, x_n]=0.
$$
The case $n=2$ is ordinary Lie algebra. The notions such as,
subalgebra and ideal can be defined as usual, see \cite{Fil}. For
$n>2$, it is proved that there is only one  simple $n$-ary Lie
algebra and the dimension of this unique simple Lie algebra is
$n+1$, see \cite{Bai} for details. This fact is a large difference
between ordinary  and $n$-ary Lie algebras, because there are lots
of simple ordinary Lie algebras. The case of $n$-ary Lie algebras,
may suggest that similarly there are a few simple polyadic groups
for $n>2$. But as we shall see, this is not true and there are many
simple $n$-ary groups (both $UAS$ and $GTS$), even more than the
ordinary simple groups. This fact, suggests us that there might be a
more general definition of $n$-ary Lie algebras which is not
discovered until today. In our opinion, the recent definition of
Filippov for $n$-ary Lie algebras is only a small piece of an
unknown algebraic structure. It may be logically true way if we look
for these general notion of real $n$-Lie algebras via studying {\em
polyadic Lie groups}.

\section{Elementary notions}
We assume that our polyadic group has the form $(G, f)=\Gf$. The
identity element of $(G, \cdot)$ will be denoted by $e$. First we
express the skew element $\overline{x}$ in terms of $x$ and
$\theta$.
\begin{lemma}
We have
$$
\overline{x}=b^{-1}\theta^{n-2}(x^{-1})\cdots
\theta^2(x^{-1})\theta(x^{-1}).
$$
\end{lemma}

\begin{proof}
Suppose $y=b^{-1}\theta^{n-2}(x^{-1})\cdots
\theta^2(x^{-1})\theta(x^{-1})$. Then we have
\begin{eqnarray*}
f(y, \stackrel{(n-1)}{x})&=&y\theta(x)\theta^2(x)\cdots \theta^{n-1}(x)b\\
                        &=&b^{-1}\theta^{n-1}(x)b\\
                        &=&b^{-1}bxb^{-1}b\\
                        &=&x.
\end{eqnarray*}
This shows that $y=\overline{x}$.
\end{proof}

\begin{definition}
Suppose $H\leq (G, f)$ and for all $x\in G$ and $h\in H$
$$
f(\overline{x},\stackrel{(n-3)}{x}, h, x)\in H.
$$
Then $H$ is called a normal polyadic subgroup of and it is denoted
by $H\unlhd (G, f)$.
\end{definition}

\begin{lemma}
Assume that $H\leq (G, f)$. Then $H$ is normal, iff for all $x\in G$
and $h\in H$, we have $\theta^{-1}(x^{-1}h)x\in H$.
\end{lemma}

\begin{proof}
Note that
\begin{eqnarray*}
f(\overline{x},\stackrel{(n-3)}{x}, h, x)&=&b^{-1}\theta^{n-2}(x^{-1})\cdots \theta(x^{-1})\theta(x)\cdots\theta^{n-3}(x)\theta^{n-2}(h)bx\\
                                       &=&b^{-1}\theta^{n-2}(x^{-1})\theta^{n-2}(h)bx\\
                                       &=&b^{-1}\theta^{n-2}(x^{-1}h)bx\\
                                       &=&b^{-1}\theta^{n-1}(\theta^{-1}(x^{-1}h))bx\\
                                       &=&b^{-1}b\theta^{-1}(x^{-1}h)b^{-1}bx\\
                                       &=&\theta^{-1}(x^{-1}h)x,
\end{eqnarray*}
and the lemma is proved.
\end{proof}

\begin{definition}
An equivalence relation $R$ over $G$ is said to be a {\em
congruence}, if

1. $\forall i:\ x_iRy_i\Rightarrow f(x_1^n)Rf(y_1^n)$,

2. $xRy\Rightarrow \overline{x}R\overline{y}$.
\end{definition}

We denote the set of all congruences of $(G, f)$ by $Cong(G, f)$.
The following theorem is proved in \cite{Dud-Shah}.

\begin{theorem}
Suppose $H\unlhd (G, f)$ and define $R=\sim_H$ by
$$
x\sim_Hy \Leftrightarrow \exists h_1, \ldots, h_{n-1}\in H: y=f(x,
h_1^{n-1}).
$$
Then $R$ is a congruence and if we let $xH=[x]_R$, (the equivalence
class of $x$), then the set $G/H=\{ xH:x\in G\}$ is an $n$-ary group
with the operation
$$
f_H(x_1H, \ldots, x_nH)=f(x_1^n)H.
$$
Further we have
$$
(G/H, f_H)=der(ret_H(G/H, f_H)),
$$
and so it is reduced.
\end{theorem}

\begin{definition}
A polyadic group $(G, f)$ is {\em group theoretically simple}, or
$GTS$, if
$$
H\unlhd (G, f)\Rightarrow |H|=1\ \ or\ \ H=G.
$$
We say that $(G, f)$ is {\em universal algebraically simple}, or
$UAS$, if
$$
R\in Cong(G, f)\Rightarrow R=\Delta\ \ or\ \ R=G\times G,
$$
where $\Delta$ denotes equality.
\end{definition}

\begin{proposition}
Every $UAS$ is also $GTS$.
\end{proposition}

\begin{proof}
Suppose $(G,f)$ is $UAS$ and $H\unlhd (G, f)$. Let $R=\sim_H$. So we
have $R=\Delta$ or $G\times G$. If $R=\Delta$, then the following
implication is true,
$$
(\exists h_1, \ldots, h_{n-1}\in H: y=f(x, h_1^{n-1}))\Rightarrow
x=y.
$$
Assume that $|H|>1$, so there exist distinct elements $h_1, h_2\in H$.
But then there is $h_3\in H$ such that $h_2=f(h_1, h_3,
\stackrel{(n-2)}{h_1})$. Therefore $h_1=h_2$, a contradiction. Hence
$|H|=1$. On the other hand, if $R=G\times G$, then for any $x, y\in
G$, we have
$$
\exists h_1, \ldots, h_{n-1}\in H: y=f(x, h_1^{n-1}).
$$
So, if we let $x\in H$, then for any $y\in G$, $ y=f(x,
h_1^{n-1})\in H$. This shows that $H=G$.
\end{proof}

A polyadic group $(G, f)$ is {\em strongly GTS}, if $H=G$ is its
only normal polyadic subgroup. The class of such polyadic groups
will be denoted by $GTS^{\ast}$. The next proposition shows that
$GTS^{\ast}$ is the more important part of $GTS$.

\begin{proposition}
If $(G, f)$ has a singleton normal polyadic subgroup, then it is
reduced.
\end{proposition}

\begin{proof}
Suppose $H=\{ u\}$ is a polyadic normal subgroup of $(G, f)$. For any $x\in
G$, the coset $xH$ is also singleton and so $xH=\{ x\}$. The map
$\delta:G\to G/H$ defined by $\delta(x)=\{x\}$ is then an
isomorphism and hence
\begin{eqnarray*}
(G, f)&\cong& (G/H, f_H)\\
      &=& der(ret_H(G/H, f_H)),
\end{eqnarray*}
and therefore $(G, f)$ is reduced.
\end{proof}

By the above proposition, if $(G, f)$ is $GTS$ but  not strong,
then $(G, f)=der(G, \cdot)$, where $(G, \cdot)$ is an ordinary
simple group. Hence it remains to determine polyadic groups which
are $GTS^{\ast}$. This will be done in the section 4.

\section{UAS and the lattice of congruences}
In this section we determine the structure of congruences of
polyadic groups and we give a necessary and sufficient condition for
a polyadic group to be $UAS$. Note that the binary group $G\times G$
has an automorphism $(x, y)\mapsto (\theta(x),\theta(y))$ which we
denote it by the same symbol $\theta$. As in the previous section,
$Cong(G, f)$ is the set of all congruences of $(G, f)$. This set is
a lattice under the operations of intersection and product
(composition). We also denote by $Eq(G)$ the set of all equivalence
relations of $G$.

\begin{theorem}
$R\in Cong(G, f)$ iff $R\in Eq(G)$ and $R$ is a $\theta$-invariant
subgroup of $G\times G$.
\end{theorem}

\begin{proof}
Let $R\in Con(G,f)$. We prove that $R$ is a $\theta$-invariant subgroup of $G\times G$. Suppose $xRy$ and
$$
x_1=x,\ y_1=y,\ x_2=y_2=e,\ \ldots, x_n=y_n=e.
$$
Now, since $f(x_1^n)Rf(y_1^n)$, we must have $xbRyb$. A similar argument with
$$
x_1=x, y_1=b^{-2},\ x_2=y_2=\cdots x_{n-1}=y_{n-1}=e,\ x_n=x,\ y_n=y,
$$
shows that $b^{-1}xRb^{-1}y$. Now, let $xRy$ and assume that
$$
x_1=y_1=e,\ x_2=x,\ y_2=y,\ x_3=y_3=e, \ldots, x_{n-1}=y_{n-1}=e,\ x_n=y_n=b.
$$
Then $f(x_1^n)Rf(y_1^n)$, which means that $\theta(x)R\theta(y)$. This shows that $\theta(R)\subseteq R$. Note that the converse is also true, i.e. if $\theta(x)R\theta(y)$, then $xRy$. This is true, since $\theta(x)R\theta(y)$ implies $\theta^{n-1}(x)R\theta^{n-1}(y)$. So we have $bxb^{-1}Rbyb^{-1}$ and therefore $xRy$ by the above argument. Summarizing what we proved above, we have
$$
xRy\Leftrightarrow \theta(x)R\theta(y).
$$
Now, assume that $xRy$ and $uRv$. Let $u^{\prime}=\theta^{-1}(u)$, and $v^{\prime}=\theta^{-1}(v)$. So $u^{\prime}Rv^{\prime}$ and again a similar argument shows that $x\theta(u^{\prime})Ry\theta(v^{\prime})$, hence $xuRyv$. This shows that $R$ is closed under the binary operation of $G\times G$. Finally we show that $xRy$ implies $x^{-1}Ry^{-1}$. Note that $xRy$ implies $\overline{x}R\overline{y}$. But
$$
\overline{x}=b^{-1}\theta^{n-2}(x^{-1})\ldots\theta^2(x^{-1})\theta(x^{-1})
$$
$$
\overline{y}=b^{-1}\theta^{n-2}(y^{-1})\ldots\theta^2(y^{-1})\theta(y^{-1}).
$$
Therefore
$$
b^{-1}\theta^{n-2}(x^{-1})\ldots\theta^2(x^{-1})\theta(x^{-1})Rb^{-1}\theta^{n-2}(y^{-1})\ldots\theta^2(y^{-1})\theta(y^{-1}),
$$
and hence
$$
\theta(\theta^{n-3}(x^{-1})\ldots\theta(x^{-1})x^{-1})R\theta(\theta^{n-3}(y^{-1})\ldots\theta(y^{-1})y^{-1}).
$$
Therefore, we conclude that
$$
\theta^{n-3}(x^{-1})\ldots\theta(x^{-1})x^{-1}R\theta^{n-3}(y^{-1})\ldots\theta(y^{-1})y^{-1}.
$$
Continuing this argument, we obtain finally $x^{-1}Ry^{-1}$. So, we proved that $R$ is a $\theta$-invariant  subgroup of $G\times G$. Note that, clearly the converse is also true, so we proved the assertion.
\end{proof}

\begin{proposition}
Let $H_R=\{ x\in G: xRe\}=[e]_R$. Then $H_R$ is a $\theta$-invariant normal subgroup of $G$ and  it is a normal polyadic subgroup of $(G,f)$ only in the case $bRe$. \\
\end{proposition}

\begin{proof}
Let $x, y\in H_R$. Then $xRe$ and $yRe$ and so $xy^{-1}Re$. Also if $x\in H_R$ and $a\in G$ then $axa^{-1}Re$ and therefore $H_R$ is a normal subgroup of $(G, \cdot)$. Moreover if $xRe$ then $\theta(x)Re$ and hence $H_R$ is $\theta$-invariant. Note that if $x_1, \ldots, x_n\in H_R$, then $f(x_1^n)Rb$ and hence $H_R\leq (G, f)$ iff $bRe$.
\end{proof}

Now, we are ready to give the necessary and sufficient condition for a polyadic group to be $UAS$.

\begin{theorem}
$(G,f)$ is UAS iff the only normal $\theta$-invariant subgroups of $(G,\cdot)$ are trivial subgroups.
\end{theorem}

\begin{proof}
Let $(G, \cdot)$ be $\theta$-simple (i.e. it has no non-trivial  normal $\theta$-invariant subgroup) and $R\in Con(G,f)$. Since $H_R$ is a $\theta$-invariant normal subgroup of $(G,\cdot)$, so $H_R=1$ or $G$. This shows that $R=\Delta$ or $G\times G$. So $(G,f)$ is UAS. Conversely,  suppose that $(G,f)$ is UAS, and let $H$ be a $\theta$-invariant normal subgroup of $(G, \cdot)$. Define
$$
R=\{ (x,y): x^{-1}y\in H\}.
$$
It is easy to see that $R\in Con(G,f)$ and so, $R=\Delta$ or $G\times G$ which implies that $H=1$ or $G$.
\end{proof}

In the previous section, we saw that $H\unlhd (G, f)$ implies that $(G/H, f_H)$ is reduced, but this is not true for $G/R$, when $R$ is an arbitrary congruence. To determine its structure, note that, since $H_R$ is $\theta$-invariant, so we can define a new automorphism
$$
\theta_R:\frac{G}{H_R}\to \frac{G}{H_R}
$$
by $\theta_R([x])=[\theta(x)]$. Let $b_R=[b]$. Then we have
\begin{eqnarray*}
f_R([x_1], \ldots,[x_n])&=&[f(x_1^n)]\\
                        &=&[x_1\theta(x_2)\ldots\theta^{n-1}(x_n)b]\\
                        &=&[x_1]\theta_R([x_2])\ldots\theta_R^{n-1}([x_n])b_R.
\end{eqnarray*}
This shows that
$$
\frac{G}{R}=der_{\theta_R, b_R}(\frac{G}{H_R}, \cdot).
$$

\begin{lemma}
Let $R\leq G\times G$. Then $R\in Eq(G)$ iff $\Delta\subseteq R$.
\end{lemma}

\begin{proof}
One side is trivial. So assume that $\Delta\subseteq R$. Let $(x,y)\in R$. We show that $(y,x)\in R$. Note that
$$
(y,x)=(x,x)(x^{-1},y^{-1})(y,y)\in R.
$$
So $R$ is symmetric. Now suppose $(x,y), (y,z)\in R$. Then
$$
(x,z)=(x,y)(y^{-1},y^{-1})(y,z)\in R.
$$
This completes the proof.
\end{proof}

\begin{corollary}
We have $Cong(G,f)=\{ R\leq_{\theta} G\times G: \Delta\subseteq R\}$.
\end{corollary}

\begin{proposition}
Let $R,Q\in Cong(G,f)$. Then the following assertions are true.

1. as subgroups of $G\times G$, we have $RQ=QR$.

2. we have $R\circ Q=RQ$.

3. we have $R\circ Q=Q\circ R$, so $Cong(G,f)$ is a modular lattice.

4. we have $H_{RQ}=H_RH_Q$ and $H_{R\cap Q}=H_R\cap H_Q$.
\end{proposition}

\begin{proof}
We know that $H_Q\unlhd G$, so $1\times H_Q\unlhd G\times G$. Let $(x,y)\in R$ and $(u,v)\in Q$. Since $R$ is a congruence, so $(xu,yu)\in R$ and also
\begin{eqnarray*}
(x,y)(u,v)&=&(x,y)(u,u)(u^{-1},u^{-1})(u,v)\\
          &=&(x u,y u)(e,u^{-1}v).
\end{eqnarray*}
Now, $(e,u^{-1}v)\in Q$, so $u^{-1}v\in H_Q$ and hence $(e,u^{-1}v)\in 1\times H_Q$. But, we have $R(1\times H_Q)=(1\times H_Q)R$, so
$$
(x,y)(u,v)=(e,w)(x^{\prime},y^{\prime}),
$$
for some $(e,w)\in Q$ and $(x^{\prime},y^{\prime})\in R$. This shows that $RQ=QR$.
To prove {\em 2}, note that
$$
R\circ Q=\{ (x,y)\in G\times G: \exists u, \ (x,u)\in Q\ and\ (u,y)\in R\}.
$$
So, if $(x,y)\in R\circ Q$, then $(x,y)=(x,u)(u^{-1}, u^{-1})(u,y)\in RQ$. Hence $R\circ Q\subseteq RQ$. The converse is also true.
The proofs are {\em 3} and {\em 4} are now clear.
\end{proof}

A congruence $R\in Cong(G, f)$ is called normal if there exists a normal polyadic group $H\unlhd (G, f)$ such that $R=\sim_H$. In what follows, we determine the normal elements of $Cong(G, f)$.

\begin{proposition}
Let $R$ be a normal congruence and $H$ be the normal polyadic subgroup corresponding to $R$. Then there exists an element $a\in G$ such that $H=aH_R$. \end{proposition}

\begin{proof}
Suppose $x, y\in H$ and $h_1, \ldots, h_{n-2}\in H$ are arbitrary elements. We know that there is $h_{n-1}\in H$ such that $y=f(x, h_1^{n-1})$. Hence $x\sim_H y$ and therefore there exists $a\in G$ such that $x, y\in [a]_R$. Hence $H\subseteq aH_R$. Now, suppose $u\in aH_R$. Then there is $x\in H$ such that $u\sim_H x$, so there are $h_1, \ldots, h_{n-1}\in H$ such that $u=f(x, h_1^{n-1})\in H$. Hence $aH_R\subseteq H$.
\end{proof}

\begin{theorem}
Let $R$ be a congruence. Then $R$ is normal iff there exists $a\in G$ such that

1. $aR\overline{a}$,

2. for all $x\in G$, we have $f(\overline{x},\stackrel{(n-3)}{x}, a, x)Ra$.
\end{theorem}

\begin{proof}
Suppose $R$ satisfies {\em 1} and {\em 2}. We first show that $H=aH_R$ is a normal polyadic subgroup. Suppose $x_i\in H$ for $1\leq i\leq n$. Then for all $i$, we have $x_iRa$ and since $aR\overline{a}$, hence $f(x_1^n)Rf(\overline{a},\stackrel{(n-1)}{a})$. This shows that $f(x_1^n)Ra$ and therefore $f(x_1^n)\in H$. On the other hand, if $x\in H$, then $xRa$ and so $\overline{x}R\overline{a}$. Now, since $aR\overline{a}$, so $\overline{x}Ra$, which proves that $\overline{x}\in H$. Therefore $H$ is a polyadic subgroup. Now, let $x\in G$ and $h\in H$. Then $hRa$ and hence, $f(\overline{x},\stackrel{(n-3)}{x}, h, x)Rf(\overline{x},\stackrel{(n-3)}{x}, a, x)$. But by {\em 2}, we have $f(\overline{x},\stackrel{(n-3)}{x}, a, x)Ra$, so
$f(\overline{x}, \stackrel{(n-3)}{x},  h, x)\in H$, proving that $H$ is a normal.

Now, we prove that $R=\sim_H$. Suppose $x\sim_H y$. Then there exist $h_1, \ldots, h_{n-1}\in H$, such that $y=f(x, h_1^{n-1})$. But, then every $h_i$ is of the form $ah_i^{\prime}$, with $h_i^{\prime}\in H_R$. Remember from 3.2 that $H_R$ is a $\theta$-invariant normal subgroup of $(G, \cdot)$. So we have $y=f(x, ah_1^{\prime}, \ldots, ah_{n-1}^{\prime})=f(x, \stackrel{(n-1)}{a})h^{\prime}$ for some $h^{\prime}\in H_R$. Note that, $f(x, \stackrel{(n-1)}{a})=x(\overline{a})^{-1}a$ by 2.1. On the other hand, $(\overline{a})^{-1}a\in H_R$ by {\em 1}. Hence
$$
y=f(x, \stackrel{(n-1)}{a})h^{\prime}\in xH_R,
$$
which shows that $xRy$. Hence, we proved that $x\sim_H y$ implies $xRy$. The converse is also true, showing that $R=\sim_H$. Therefore $R$ is a normal congruence.

Now suppose $R$ is normal. So $R=\sim_H$ for some normal polyadic subgroup $H$. By 3.7, there is  $a\in G$ such that $H=aH_R$. We prove that $a$ satisfies {\em 1} and {\em 2} above. Let $x_1, \ldots, x_n\in H$. Then for all $i$, we have $x_iRa$, and hence $f(x_1^n)Rf(\stackrel{(n)}{a})$. Since $H$ is a polyadic subgroup, so $f(x_1^n)Ra$ and hence $aRf(\stackrel{(n)}{a})$. Now, using 2.1, it is easy to see that $f(\stackrel{(n)}{a})=a(\overline{a})^{-1}a$. Therefore, $f(\stackrel{(n)}{a})a^{-1}\in H_R$ implies that $a(\overline{a})^{-1}aa^{-1}\in H_R$. This shows that $a(\overline{a})^{-1}\in H_R$. So $aR\overline{a}$, proving {\em 1}. The proof of {\em 2} is similar.
\end{proof}

Form the above theorem, one can deduce that if $R, Q\in Cong(G, f)$ with $R$ normal and $R\subseteq Q$, then $Q$ is also normal. We can restate the above theorem as in the following form.

\begin{corollary}
A congruence $R$ is normal iff there exists an element $a\in G$ such that

1. $aR\overline{a}$,

2. for all $x\in G$, $\theta^{-1}(x^{-1}a)xRa$.
\end{corollary}

\section{GTS and normal polyadic subgroups}
This section is devoted to $GTS$ polyadic groups. Again, we assume that $(G,f)=\Gf$ is an $n$-ary group. For $u\in G$, define a new binary operation on $G$ by $x\ast y=xu^{-1}y$. Then $(G, \ast)$ is an isomorphic copy of $(G,\cdot)$ and the isomorphism is the map $x\mapsto xu$. We denote this new group by $G_u$. Its identity is $u$ and the inverse of $x$ is $ux^{-1}u$, which we denote it by $x^{-u}$. We define an automorphism of $G_u$ by $\psi_u(x)=u\theta(x)\theta(u^{-1})$. It can be easily checked that this is actually an automorphism of $G_u$.

\begin{theorem}
We have $H\unlhd (G, f)$ iff there exists an element $u\in H$ such that

1. $H$ is a $\psi_u$-invariant normal subgroup of $G_u$,

2. for all $x\in G$, we have $\theta^{-1}(x^{-1}u)x\in H$.
\end{theorem}

\begin{proof}
We first determine the structure of polyadic subgroups, using 1.2. Suppose $H\leq (G, f)$. We denote the restriction of $f$ to $H$ by $f$, so there is a binary operation on $H$, say $\ast$, an automorphism $\psi$ and an element $c\in H$, such that
$$
(H, f)=der_{\psi, c}(H, \ast).
$$
The inclusion map $j:H\to G$ is a polyadic homomorphism, hence by 1.2, there is an element $u\in G$ and an ordinary homomorphism $\phi:(H, f)\to (G,f)$, with the properties

i-\ \  $j=R_u\phi$,

ii-\  $f(\stackrel{(n)}{u})=\phi(c)u$,

iii- $\phi\psi=I_u\theta\phi$.

From i, we deduce that for any $x\in H$, $\phi(x)=xu^{-1}$, and so by ii, we have $f(\stackrel{(n)}{u})=c$. Moreover, since $\phi$ is an ordinary homomorphism, so using $\phi(x\ast y)=\phi(x)\phi(y)$ and $\phi(x)=xu^{-1}$, we obtain $x\ast y=xu^{-1}y$. Finally, by iii, we have $\psi(x)=u\theta(x)\theta(u^{-1})$, and therefore we must have $(H, \ast)\leq G_u$. Further, $H$ is invariant under $\psi_u$ and hence $\psi=\psi_u|_H$. So, we proved that  $H$ is a polyadic subgroup of $(G, f)$ iff there exists an element $u$ such that $H$ is a $\psi_u$-invariant subgroup of $G_u$. Now, suppose such an $H$ is  normal in $(G, f)$. We show that $H\unlhd G_u$, equivalently
$$
x^{-u}\ast h\ast x\in H
$$
for all $x\in G_u$ and $h\in H$. Note that this last statement is equivalent to $ux^{-1}hu^{-1}x\in H$. Since $H$ is a normal polyadic subgroup, so by 2.3,
$$
\theta^{-1}(x^{-1}h)x, \theta^{-1}(x^{-1}u)x\in H.
$$
But $H$ is $\psi_u$-invariant, so
$$
\psi_u(\theta^{-1}(x^{-1}h)x), \psi_u(\theta^{-1}(x^{-1}u)x)\in H.
$$
Therefore the following element also belongs to $H$,
\begin{eqnarray*}
\psi_u(\theta^{-1}(&x^{-1}&h)x)\ast\psi_u(\theta^{-1}(x^{-1}u)x)^{-u}\\
                             &=&ux^{-1}h\theta(x)\theta(u^{-1})u^{-1}u(ux^{-1}u\theta(x)\theta(u^{-1}))^{-1}u\\
                             &=&ux^{-1}hu^{-1}x.
\end{eqnarray*}
This shows that $H\unlhd G_u$. Note that the following is also automatically holds:
$$
\forall x\in G: \theta^{-1}(x^{-1}u)x\in H.
$$
Conversely, suppose there is a $u\in G$ such that $H$ is a $\psi_u$-invariant normal subgroup of $G_u$ and
$$
\forall x\in G: \theta^{-1}(x^{-1}u)x\in H.
$$
We show that $H$ is a normal polyadic subgroup. The equality
$$
\psi_u(\theta^{-1}(x^{-1}h)x)\ast \psi_u(\theta^{-1}(x^{-1}u)x)^{-u}=x^{-u}\ast h\ast x
$$
shows that $\psi_u(\theta^{-1}(x^{-1}h)x)\in H$, and since $H$ is invariant under $\psi_u$, so $\theta^{-1}(x^{-1}h)x\in H$. Therefore, $H$ is a normal polyadic subgroup.
\end{proof}

\begin{lemma}
Suppose $u\in G$ is an arbitrary element. Then $H$ is a
$\theta$-invariant normal subgroup of $(G,\cdot)$ iff $Hu$ is
$\psi_u$-invariant normal subgroup of $G_u$.
\end{lemma}

\begin{proof}
Suppose $H\unlhd_{\theta} G$. Then it can be checked that $Hu$ is a
subgroup of $G_u$. For any $x\in G$ and $h\in H$ we have
\begin{eqnarray*}
x^{-u}\ast hu\ast x&=&(ux^{-1}u)u^{-1}(hu)u^{-1}x\\
                   &=&ux^{-1}hx,
\end{eqnarray*}
which is clearly an element of  $Hu$. So $Hu\unlhd G_u$. Also
$$
\psi_u(hu)=u\theta(hu)\theta(u^{-1})=u\theta(u),
$$
which shows that $Hu$ is $\psi_u$-invariant. Conversely, suppose $H$
is a $\psi_u$-invariant normal subgroup of $G_u$. We show that
$Hu^{-1}$ is a $\theta$-invariant normal subgroup of $G$. We have
$$
xhu^{-1}x^{-1}=(xu)\ast h\ast (xu)^{-u}u^{-1},
$$
which belongs to $Hu^{-1}$. So $Hu^{-1}\unlhd G$. Similarly,
$\theta(hu^{-1})=u^{-1}\psi_u(h)$ belongs to $u^{-1}H=e\ast H=H\ast
e=Hu^{-1}$. Hence $Hu^{-1}$ is $\theta$-invariant.
\end{proof}

Suppose $K$ is a $\theta$-invariant normal subgroup of $(G, \cdot)$.
Then $\theta$ induces an automorphism of $G/K$ which we denote it by
$\theta_K$ in what follows.

\begin{corollary}
Let $H\unlhd(G, f)$. Then there exists an element $u$ such that
$K=H\cdot u^{-1}$ is a $\theta$-invariant normal subgroup of $G$ and
$\theta_K$ is an inner automorphism. The converse is also true.
\end{corollary}

\begin{proof}
First, we notice that $H\cdot u^{-1}$ is not a polyadic coset, but it is the set $\{ hu^{-1}: h\in H\}$.
Suppose $H\unlhd (G, f)$. By 4.1, there is an element $u\in H$ such
that $H$ is a $\psi_u$-invariant normal subgroup of $G_u$ and for
any $x\in G$, we have $\theta^{-1}(x^{-1}u)x\in H$. Let $K=H\cdot u^{-1}$.
By the above lemma $K$ is a $\theta$-invariant normal subgroup of
$G$. Now, $\theta^{-1}(x^{-1}u)x\in H$ and $H$ is
$\psi_u$-invariant, so $x^{-u}\ast \psi_u(x)\in H$. Therefore
$\psi_u(x)\ast H=x\ast H$. Since $H$ is normal in $G_u$, we have
$H\ast \psi_u(x)=H\ast x$ and this is equivalent to $K\psi_u(x)=Kx$.
Now $K$ is a normal subgroup of $G$, and hence $\psi_u(x)K=xK$. So
$\theta(xu^{-1})K=u^{-1}xK$. If we put $y=xu^{-1}$, then
$\theta(y)K=u^{-1}yuK$ and this proves that $\theta_K$ is an inner
automorphism.

Conversely, suppose $K$ is a $\theta$-invariant normal subgroup of
$G$ and $\theta_K=I_{uK}$. Then by a similar argument one can prove
that $H=Ku^{-1}\unlhd (G, f)$.
\end{proof}

We saw before that if $(G, f)$ is a non-strong $GTS$, then it is
reduced. So we determine when a polyadic group belongs to the class
$GTS^{\ast}$.

\begin{theorem}
A polyadic group $(G, f)$ is $GTS^{\ast}$ iff whenever $K$ is a
$\theta$-invariant normal subgroup of $(G, \cdot)$ with $\theta_K$ inner,
then $K=G$.
\end{theorem}

\begin{proof}
First let $(G,f)$ be $GTS^{\ast}$ and $K$ be a
$\theta$-invariant normal subgroup of $G$ with $\theta_K$ inner. Then by the above corollary, there is a $u$ such that $H=Ku^{-1}\unlhd (G, f)$. Hence $Ku^{-1}=G$ and so $K=G$. To prove the converse, suppose $H\unlhd (G, f)$. So there is a $u$ such that $K=H\cdot u^{-1}$ satisfies our hypothesis. Therefore $K=G$ and hence $H=G$, proving that $(G, f)$ is $GTS^{\ast}$.
\end{proof}

As a final remark, the reader most notice that the binary groups $G_u$ which we used in this section, are in fact retract of $(G, f)$ in the case when $u$ is the skew for some other element. Suppose $u=\overline{a}$. In the group $ret_a(G, f)$, as we said in the introduction, the identity is $\overline{a}$ and if we define $c=f(\stackrel{(n)}{\overline{a}})$ and $\phi(x)=f(\overline{a}, x, \stackrel{(n-2)}{a})$, then  by \cite{Sok}, we have
$$
(G, f)=der_{\phi, c}(ret_a(G, f), \ast).
$$
It is easy to see that $x\ast y=x(\overline{a})^{-1}y$ and hence $G_u=ret_a(G, f)$. But, note that in general, it is false to say that every $u$ is equal to the skew element of some $a$; there are polyadic groups in which the skew elements of any $x$ and $y$ are equal. So in the general case $G_u$ is not equal to any retract.

\end{document}